\documentclass{article}

\usepackage[paper=a4paper,total={18.5cm, 25.5cm}]{geometry}
\usepackage[affil-it]{authblk}
\usepackage{hyperref}

\usepackage{graphicx} % Required for inserting images

\usepackage{amsmath}
\usepackage{amssymb}
\usepackage{amsthm, color, bm}
\usepackage{mathtools}
\usepackage{amsfonts}
\usepackage{mathrsfs}
\usepackage{framed}

\newtheorem{lemma}{Lemma}
\newtheorem{theorem}{Theorem}
\newtheorem{remark}{Remark}

\newtheorem{problem}{Problem}
\usepackage{mathrsfs}

\newcommand{\R}{\mathbb{R}}

\newcommand{\W}{W}

\DeclareMathOperator{\Div}{div}

\DeclareMathOperator{\Grad}{grad}
\DeclareMathOperator{\Grads}{\Grad_S}

\DeclareMathOperator{\Kernel}{Ker}
\DeclareMathOperator{\Image}{Im}

\DeclareMathOperator{\st}{s.t.}
\DeclareMathOperator{\Td}{T_0}
\DeclareMathOperator{\Tn}{T_\nu}

\newcommand{\RN}{\mathbb{R}^N}

\newcommand{\MN}{\mathbb{M}^N}
\newcommand{\SN}{\mathbb{S}^N}
\newcommand{\AN}{\mathbb{A}^N}

\newcommand{\rbm}{R}

\newcommand{\LpRN}{L^p(\Omega;\RN)}
\newcommand{\LqRN}{L^q(\Omega;\RN)}
\newcommand{\LpSN}{L^p(\Omega;\SN)}
\newcommand{\LqSN}{L^q(\Omega;\SN)}
\newcommand{\LpMN}{L^p(\Omega;\MN)}

\newcommand{\LoneMN}{L^1(\Omega;\SN)}
\newcommand{\LinfMN}{L^\infty(\Omega;\MN)}

\newcommand{\DpSN}{W^{\Div,p}(\Omega;\SN)}
\newcommand{\DpzSN}{W^{\Div,p}_0(\Omega;\SN)}
\newcommand{\DqSN}{W^{\Div,q}(\Omega;\SN)}
\newcommand{\DqzSN}{W^{\Div,q}_0(\Omega;\SN)}
\newcommand{\WpRN}{W^{1,p}(\Omega;\RN)}
\newcommand{\WqRN}{W^{1,q}(\Omega;\RN)}
\newcommand{\WpzRN}{W^{1,p}_0(\Omega;\RN)}
\newcommand{\WqzRN}{W^{1,q}_0(\Omega;\RN)}

\newcommand{\WoneinfRN}{W^{1,\infty}(\Omega;\RN)}

\newcommand{\testRN}{\mathcal{D}(\Omega;\RN)}
\newcommand{\distRN}{\mathcal{D}'(\Omega;\RN)}

\NewDocumentCommand{\sprod}{O{} O{} m m}{
  \langle #3, #4 \rangle_{#1, #2}
}
\newcommand\norm[2][]{\|#2\|_{#1}}

\title{Data‑driven stress problem under purely \\normal homogeneous Neumann boundary conditions}

\author[1]{\small Cristian G. Gebhardt}
\author[2]{\small Kundan Kumar\thanks{Corresponding author.\textit{E-mail-address}:\href{mailto:kundan.kumar@uib.no}{kundan.kumar@uib.no}}}
\author[2]{\small Florin A. Radu}

\affil[1]{\small University of Bergen, Geophysical Institute and Bergen Offshore Wind Centre (BOW), All\'{e}gaten 70, 5007 Bergen, Norway}
\affil[2]{\small University of Bergen, Department of Mathematics and Center for Modeling of Coupled Subsurface Dynamics (CSD), All\'{e}gaten 41, 5007 Bergen, Norway}

\date{}

\begin{document}

\maketitle

\begin{abstract}
\noindent Data‑Driven Continuum Mechanics -- the continuous counterpart of Data-Driven Computational Mechanics -- is a modern paradigm that enhances classical continuum mechanics by incorporating finite sets of experimental material data directly, avoiding any form of constitutive modeling. Despite recent progress, its analytical foundations remain at an early stage. In this work, we establish a rigorous functional‑analytic framework for the data-driven stress problem under purely homogeneous normal Neumann boundary conditions. The problem is formulated as finding a stress field (satisfying the balance of linear and angular momenta and the boundary conditions) that is closest, in an $L^p$-sense, to an auxiliary stress field that is simultaneously sought and locally resembles a finite discrete set of experimental stress states. Our analysis relies on two key ingredients. First, the divergence operator induces a topological isomorphism between the space of symmetric stress fields modulo its kernel and the space of loads balanced by rigid-body motions, ensuring the existence of an equilibrated response. Second, the finiteness of the material data set guarantees proximinality in the stress space, which in turn yields a complete existence and uniquenes theory for solution equivalence classes. Together, these two properties provide a rigorous mathematical foundation for the data‑driven stress problem under purely homogeneous normal Neumann boundary conditions.
\vspace{5mm}

\vspace{5mm}


\end{abstract}

\section{Introduction}

The classical approach in continuum mechanics relies on the availability of well‑defined phenomenological constitutive laws. However, for complex or non‑standard materials, such laws often introduce significant modeling uncertainties. To overcome this limitation, one can adopt a Data‑Driven Continuum Mechanics framework\footnote{Extending the term coined by Conti, Müller and Ortiz in \cite{ContiMuellerOrtiz2018,ContiMuellerOrtiz2020}.}, which replaces traditional constitutive equations with a direct search within a discrete set of experimental material data. In this setting, the goal is to identify states that satisfy the fundamental physical constraints -- balance of linear and angular momenta, geometrical compatibility, and prescribed boundary conditions -- while minimizing the distance to auxiliary fields that locally resemble the available experimental data \cite{ContiMuellerOrtiz2018,ContiMuellerOrtiz2020,GebhardtSteinbach2026}.\\

In this contribution, we develop a rigorous functional‑analytic foundation for, perhaps, the simplest specialization of the general Data‑Driven Continuous Mechanics framework, focusing exclusively on symmetric stress fields under purely normal homogeneous Neumann boundary conditions. Working within Sobolev and Lebesgue spaces, we establish the conditions required to ensure both the existence and uniqueness of the resulting equilibrated stress field, as well as the existence of an auxiliary stress field that locally resembles the available experimental data. To this end, we begin by introducing the essential analytical ingredients needed to formulate the problem under consideration. We then proceed to define the problem precisely and present the corresponding existence theorem, thereby establishing the mathematical framework in which the remainder of this contribution is developed.\\

Let $\Omega \subset\mathbb{R}^N$ ($N \geq 2$) be a bounded Lipschitz domain with boundary $\Gamma := \partial \Omega$. Let $p,q \in (1,\infty)$ be Hölder conjugate exponents, i.e., 
$1/p+1/q = 1.$
Let $\mathbb{M}^N$ be the space of real $N \times N$ matrices and consider its canonical direct-sum decomposition $\mathbb{M}^N = \mathbb{S}^N\oplus\mathbb{A}^N$, where $\mathbb{S}^N$ and $\mathbb{A}^N$ are mutually orthogonal subspaces of symmetric and anti-symmetric matrices, respectively. The orthogonality is taken with respect to the Frobenius inner product. Moreover, 
$\dim(\MN) = N^2$, $\dim(\SN) = N(N+1)/2$ and $\dim(\AN) = N(N-1)/2.$\\
Let
$$\DpSN:= \lbrace s\in\LpSN:\Div s\in\LpRN\rbrace$$
and
$$\DpzSN:= \lbrace s \in \DpSN: s\nu = 0 \in W^{-1/p,p}(\Gamma; \mathbb{R}^N)\rbrace,$$
where $\nu \in S^{N-1}$ is the outward unit normal on $\Gamma$.\\
Let 
$${\mathscr S} := \lbrace \tilde{s} \in \LpSN: \tilde{s}(\omega) = \mathfrak{s}_m \in \mathbb{S}^N~\text{for a.e.}~\omega \in \Omega,~\mathrm{for~some}~\mathfrak{s}_m \in \mathfrak{S}\rbrace$$
and $\mathfrak{S} = \lbrace \mathfrak{s}_m\rbrace_{m\in\mathcal{M}} \subset \mathbb{S}^N$ is a given finite set of symmetric stress states.

Introduce the space of infinitesimal rigid motions
$$
\rbm_q := \lbrace\ w:\Omega\to\mathbb{R}^N : w(\omega) = A\omega + b,\ A \in \mathbb{A}^N,\ b \in \mathbb{R}^{N}\rbrace\cap \LqRN,
$$
and its annihilator with respect to the $\LpRN-\LqRN$ dual pairing
$$
\rbm^\circ_q:= \lbrace f \in\LpRN:\sprod[\LpRN][\LqRN]{f}{w}=0~\forall~w\in \rbm_q \rbrace.
$$ 
Assuming a Helmholz-Weyl decomposition (to be elaborated later), let $\Pi$ denote the projector that maps any $v\in\LpSN$ to its divergence-free component such that $v = \Pi v+\Grads u$ for some $u\in\WpRN$, understood up to infinitesimal rigid body motions.\\

We now define the data-driven stress problem under purely homogeneous Neumann boundary conditions ($DDSPN_0$).

\begin{framed}
\begin{problem}[$DDSPN_0$]
  Given $f\in\rbm^\circ_p$,
  consider the variational problem
$$
\smash[b]{\inf_{(s,\tilde{s})\in\DpzSN\times\LpSN}}\quad\frac{1}{p}\|s-\tilde{s}\|^p_{\LpSN}~~\st~~\Div s+f = 0,\quad\Pi s = 0, \quad \tilde{s}\in{\mathscr S}.
$$
\label{Prob:DDSPN0}
\end{problem}
\end{framed}

The constraint $\Div s+f = 0$ enforces the balance of linear and angular momenta, while the constraint $\Pi s = 0$ suppresses the divergence-free component of the stress field, thereby selecting the canonical representative within the corresponding equivalence class. The auxiliary field $\tilde{s}$ is restricted to the set $\mathscr{S}$. The objective functional measures the distance between these  stress fields in the $\LpSN$-norm. The interplay between the equilibrium constraint and the discrete nature of the data requires establishing both existence of admissible stress fields and the well-posedness of the closest‑point projection onto $\mathscr{S}$.

\begin{framed}
\begin{theorem}[$DDSPN_0$]
    Let $f \in \rbm^\circ_p$. Then $DDSPN_0$ admits at least one solution
    $$(s_f, \tilde{s}) \in \DpzSN\times \mathscr{S}$$
    satisfying
    $$\Div s_f + f = 0,\quad \Pi s_f = 0.$$
    Moreover, the equilibrated stress field $s_f$ is uniquely determined, whereas the auxiliary field $\tilde{s}$ not need to be unique.
\label{Thm:DDSPN0}
\end{theorem}
\end{framed}

The proof of this main theorem rests on two fundamental ingredients. First, the divergence operator induces a topological isomorphism between the quotient space $\DpzSN/\Kernel\left(\Div|_{\DpzSN}\right)$ and the space $\rbm^\circ_p$, which guarantees the existence of admissible stress fields for every balanced loading field $f$. The additional constraint $\Pi s = 0$ then selects the unique representative of minimal norm within the corresponding solution equivalence class.  Second, the finiteness of the stress data set $\mathfrak{S}$ guarantees that the associated admissible auxiliary stress fields $\mathscr{S}$ are proximinal in $\LpSN$. Consequently, the distance-minimization problem admits solutions.\\

The organization of this manuscript is as follows. Section 2 introduces the mathematical machinery required for the analysis, and Section 3 presents the proof of the main result.

\section{Preliminaries}

\begin{framed}   
\begin{theorem}[Poincaré's Inequality \cite{ErnGuermond2004}]
Let %$\Omega$ and $\Gamma$ be as indicated above and let
$1 \leq p \leq \infty$. Then there exists a constant $C = C(\Omega,p) > 0$
such that for every $u\in\WpRN$,
$$
\|u - u_{\Omega}\|_{\LpRN}
\le
C \, \|\nabla u\|_{\LpRN},
\qquad\text{where}\qquad
u_{\Omega}
:=
\frac{1}{|\Omega|}\int_{\Omega}u(\omega)~d\omega.
$$    
\end{theorem}
\end{framed}

\begin{framed}
\begin{theorem}[Korn's Second Inequality \cite{ErnGuermond2004}]
Let %$\Omega$ and $\Gamma$ be as indicated above and let
$1 < p < \infty$.
Then there exists a constant $C = C(\Omega,p) > 0$ such that for every
$u\in\WpRN$,
$$
\inf_{w\in\rbm_p}
\|\Grad u - \Grad w\|_{\LpMN}
\;\le\;
C \,\|\Grads u\|_{\LpMN}.
$$
\end{theorem}
\end{framed}

\begin{remark}[Failure of Korn's inequality for $p=1$ \cite{Ornstein1962}]
There exists no constant $C>0$ such that
$$
\inf_{w\in\rbm_1}
\|\Grad u - \Grad w\|_{\LoneMN}
\;\le\;
C \,\|\Grads u\|_{\LoneMN}
~\forall~ u \in \testRN.
$$
In other words, the full gradient cannot be controlled by the symmetric
part of the gradient, even modulo rigid-body motions.  
In particular, there exist smooth vector fields for which
$
\|\Grads u\|_{\LoneMN} \to 0
$
while
$
\|\Grad u\|_{\LoneMN} \to \infty
$.
\end{remark}

\begin{remark}[Failure of Korn's inequality for $p=\infty$ \cite{Mityagin1958,deLeeuwMirkil1964}]
There exists no constant $C>0$ such that
$$
\inf_{w\in\rbm_\infty}
\|\Grad u - \Grad w\|_{\LinfMN}
\;\le\;
C \,\|\Grads u\|_{\LinfMN}
~\forall~u\in \WoneinfRN.
$$
Again, the full gradient cannot be controlled by its symmetric
part, even modulo rigid-body motions.
In particular, there exist smooth vector fields for which
$
\|\Grads u\|_{\LinfMN} \to 0
$
while
$\|\Grad u\|_{\LinfMN} \not\to 0
$.
\end{remark}

\begin{framed}
\begin{lemma}[Traces in Sobolev spaces \cite{ContiMuellerOrtiz2020}]
%Let $\Omega$, $\Gamma$, $p$ and $q$ be as introduced above. 
Since $C^1(\overline{\Omega},\RN)$ is dense in $\WpRN$ and $C^1(\overline{\Omega},\SN)$ is dense in $\DqSN$, the following hold:
\begin{enumerate}
\item There exists a surjective bounded linear map
$\Td:\WpRN\to\W^{1-1/p,p}(\Gamma,\RN)$
such that $\Td u = u|_\Gamma~\forall~u\in C^1(\overline{\Omega},\RN)$ and $\Kernel(\Td)= \WpzRN$.
\item There exists a surjective bounded linear map
$\Tn:\DqSN\to W^{-1/q,q}(\Gamma,\RN)$
such that $\Tn s = (s \nu)|_\Gamma~\forall~s\in C^1(\overline{\Omega},\SN)$,
where $\nu$ denotes the outer unit normal on $\Gamma$, and
$\Kernel(\Tn)= \DqzSN$.
\end{enumerate}    
\end{lemma}
\end{framed}

\begin{framed}
\begin{theorem}[Green’s identity for the symmetric gradient \cite{Grisvard1985,Ciarlet1988}]
For every $s\in\DqSN$ and $u\in\WpRN$, the following Green's identity holds
$$
\sprod[\LqRN][\LpRN]{\Div s}{u}+\sprod[\LqSN][\LpSN]{s}{\Grads u} = \sprod[W^{-1/q,q}(\Gamma,\RN)][W^{1-1/p,p}(\Gamma,\RN)]{\Tn s}{\Td u}. 
$$
\end{theorem}
\end{framed}

The following result is a regular adaptation of the duality argument from Hilbert spaces, as given in \cite{gudoshnikov2025stress} to the $W^p$ setting. The argument relies crucially on reflexivity, which holds here for $1 < p < \infty$.

\begin{framed}
\begin{theorem}[Duality $\Grads$--$\Div$]
Let $\DqzSN$ be defined as the subspace of $\DqSN$ such that
$$
\sprod[W^{-1/q,q}(\Gamma,\RN)][W^{1-1/p,p}(\Gamma,\RN)]{T_\nu s}{T_0 u}=0~\forall~u\in\WpRN. 
$$
Then, by considering $\LpRN$ and $\LpSN$ and their corresponding duality pairings and regarding $\Grads|_{\WpRN}$ as an unbounded operator that maps dense subspaces of $\WpRN$ of $\LpRN$ to $\LpSN$,
$$
\Grads|_{\WpRN}^* = -\Div|_{\DqzSN}
$$
holds in the the sense of distributions.    
\end{theorem}
\end{framed}

\begin{framed}
\begin{theorem}[Helmholtz-Weyl decomposition]
Introduce the subspaces
$$
M :=\lbrace\Grads u : u \in\WpRN\rbrace = \Image(\Grads|_{\WpRN})% = \Image(B)
$$
and
$$
N :=\lbrace s \in\DpzSN:\Div s = 0\in\LpRN\rbrace = \Kernel\left(\Div|_{\DpzSN}\right).% = \Kernel(B^*).
$$
Then, $M$ and $N$ are closed in $\LpSN$, and $M\oplus N = \LpSN$.  
\end{theorem}
\end{framed}

\begin{proof}
\textbf{Part I -- Closedness of $M$.}
Let $\{s_n\} \subset M$ with $s_n \to s$ in $\LpSN$.  
For each $n$, choose $u_n\in\WpRN$ such that
$
s_n = \Grads u_n.
$
Replace $u_n$ by its equivalence class $[u_n]$ in the quotient space $\in\WpRN/\rbm_p$,
so that symmetric gradients are unchanged.  
In this quotient setting, Korn's second inequality applies directly.  
Hence, there exists $C>0$ such that
$$
\|\Grad u_n\|_{\LpMN}
\leq
C\|\Grads u_n\|_{\LpSN}
=
C\|s_n\|_{\LpSN} .
$$
Since $s_n \to s$ in $\LpSN$, the sequence $\{\Grads u_n\}$ is bounded in $\LpMN$.  
By Poincaré's inequality, $\{u_n\}$ is bounded in $\WpRN$.
Because $1 < p < \infty$, the space $\WpRN$ is reflexive.  
Hence there exists $\{u_{n_k}\}$ and $u\in\WpRN$ such that
$
u_{n_k}\rightharpoonup u\in\WpRN.
$
In particular, $\Grads u_{n_k}\rightharpoonup\Grads u\in\LpSN$.
But since $s_n = \Grads u_n\to s$ in $\LpSN$, the weak limit must coincide with the strong limit. Thus
$
\Grads u = s.
$
Hence $s\in M$, proving that $M$ is closed in $\LpSN$.

\medskip

\textbf{Part II -- Closedness of $N$.}
The space $\DpzSN$ is a Banach space when equipped with the graph norm
$$
\|s\|_{\Div}
=
\|s\|_{\LpSN}+\|\Div s\|_{\LpRN}~\forall~s\in\DpzSN.
$$
Since $N = \Kernel\left(\Div|_{\DpzSN}\right)$ is the kernel of a bounded linear operator between Banach spaces, it follows that $N$ is closed.
\medskip

\textbf{Part III -- The decomposition.} To show that $M$ and $N$ span $\LpSN$, we observe that if $s\in M\cap N$, then
$$
\sprod[\LqSN][\LpSN]{\Grads\varphi}{s}
=
-\sprod[\testRN][\distRN]{\varphi}{\Div s}
=
0
~\forall~\varphi \in \testRN.
$$
Since $s \in M$, there exists a unique class $[u]\in\WpRN/\rbm_p$ such that $s = \Grads[u]$. By the density of $\testRN$ in $\WqzRN$, the above identity extends to
$$\sprod[\LqSN][\LpSN]{\Grads\varphi}{\Grads u} = 0~\forall~\varphi\in\WqzRN.$$
Applying Korn's first inequality combined with Poincaré's inequality, we have
$$
\|\varphi\|_{\WqRN}\leq C_q\|\Grads\varphi\|_{\LqSN}~\forall~\varphi\in\WqzRN.
$$
Applying Korn's second inequality on the quotient space, we have
$$
\|[u]\|_{\WpRN}\leq C_p\|\Grads[u]\|_{\LpSN}~\forall~[u]\in\WpRN/\rbm_p.
$$
Then, for any nonzero $\varphi\in\WqzRN$, we have the following progression
\begin{align*}
\sup_{0 \neq [u]\in\WpRN/\rbm_p}\frac{\sprod[\LqSN][\LpSN]{\Grads\varphi}{\Grads[u]}}{\|[u]\|_{\WpRN}}
%&=
%\sup_{0 \neq [u]\in\WpRN/\rbm_p}\frac{|\sprod[\LqSN]%[\LpSN]{\Grads\varphi}{\Grads[u]}|}{\|[u]\|_{\WpRN}}\\
&\geq
\sup_{0 \neq [u]\in\WpRN/\rbm_p}\frac{|\sprod[\LqSN][\LpSN]{\Grads\varphi}{\Grads[u]}|}{C_p\|\Grads[u]\|_{\LpSN}}\\
&\geq
\sup_{0 \neq s\in\LpSN}\frac{|\sprod[\LqSN][\LpSN]{\Grads\varphi}{s}|}{C_p\|s\|_{\LpSN}}\\
&\geq
\frac{1}{C_p}\|\Grads\varphi\|_{\LqSN}\\
&\geq
\frac{1}{C_p C_q}\|\varphi\|_{\WqRN},
\end{align*}
which yields the $\inf-\sup$ condition 
$$
\inf_{0 \neq \varphi\in\WqzRN}\sup_{0 \neq [u]\in\WpRN/\rbm_p}\frac{\sprod[\LqSN][\LpSN]{\Grads\varphi}{\Grads[u]}}{\|\varphi\|_{\WqRN}\|[u]\|_{\WpRN}}\geq\frac{1}{C_p C_q} > 0.
$$
This result establishes weak coercivity and injectivity for the gradient operator; consequently, since the pairing is zero for all $\varphi\in\WqzRN$, the $\inf-\sup$ bound forces $\Grads[u] = 0$. Hence $s = \Grads[u] = 0$, proving $M\cap N = \{0\}$.

In addition, since $M$ and $N$ are closed, the Closed Range Theorem establishes duality through the annihilator, i.e.,  $M = N^\circ$ and $N^\circ = M$. It follows that $\{0\}= M\cap N = N^\circ\cap M^\circ = (M+N)^\circ$.

To show that $M+N$ is closed, let $\lbrace s_n \rbrace \in M+N$ be such that $s_n\to s \in \LpSN$. Writing $s_n = s_{M_n}+s_{N_n}$ (where $s_{M_n}\in M$ and $s_{N_n}\in N$) and noting that $\sprod[\LpSN][\LqSN]{s_{N_n}}{\Grad\varphi}=0$ for any $\varphi\in\WqzRN$. Observe that
\begin{align*}
|\sprod[\LpSN][\LqSN]{s_{M_n}}{\Grad\varphi}|
&=
|\sprod[\LpSN][\LqSN]{s_n}{\Grad\varphi}|\\
&\leq
\norm[\LpSN]{s_n}\norm[\LqSN]{\Grads\varphi}\\
&\leq
\norm[\LpSN]{s_n}\norm[\WqRN]{\varphi}
\end{align*}
implies $\norm[\LpSN]{s_{M_n}}\leq C\norm[\LpSN]{s_n}$. 
Since $\lbrace s_n \rbrace$ is convergent, it is bounded. Thus, $\lbrace s_{M_n}\rbrace$ and $\lbrace s_{N_n} = s_n-s_{M_n}\rbrace$ are bounded in the reflexive space $\LpSN$ ($1<p<\infty$). By closedness (and weak closedness) of $M$ and $N$, there exist limits $s_M\in M$ and $s_N \in N$ such that $s = s_M+s_N \in M+N$. Consequently, the sum $M+N$ is closed.

Due to the triviality of the annihilator $(M+N)^\circ$ and the closedness of $M+N$, we conclude that $\LpSN = M \oplus N$.
\end{proof}

\section{Main results}

\begin{framed}
\begin{theorem}
Let $f \in \rbm^\circ_p$. Then,
$
\Div s + f = 0
$
in $\Omega$ and $s\nu = 0$ on $\Gamma$ has a unique solution equivalent class $[s]\in\DpzSN/\Kernel(\Div|_{\DpzSN})$ and this equilibrated stress can be represented as $[s] = \Grads [v]$ for a unique equivalence class $[v]\in\WpRN/\rbm_p$. Moreover, the representative element $s_f\in[s]$ satisfying $\Pi s_f = 0$, is the unique element of $[s]$ with minimal $\LpSN$-norm.
\end{theorem}
\end{framed}

\begin{proof}
\textbf{Part I -- Existence.}
The range of $-\Div:\DpzSN\to\LpRN$ is the annihilator of the kernel of its adjoint. Due to the reflexivity of Lebesgue spaces for $1 < p < \infty$, we have $(\LpRN)^* \cong \LqRN$ and 
$$-\Div|_{\DpzSN}^*
=
\Grads|_{\WqRN}^{**}
=
\Grads|_{\WqRN}.$$
Since $\Kernel\left(\Grads|_{\WqRN}\right) = \rbm_p$, the Closed Range Theorem implies
$
\Image\left(-\Div|_{\DpzSN}\right)
=
\rbm^\circ_p\subset\LpRN.
$
Since $f\in\rbm^\circ_p$, there exists $s\in\DpzSN$ such that $\Div s + f = 0$.
\medskip

\textbf{Part II -- The solution as a gradient.}
Since $\LpSN = M\oplus N$, the solution from the previous part decomposes $s = s_M+s_N$, where $s_M\in M$ and $s_N\in N$. By the definition of annihilator $N = \Kernel\left(\Div|_{\DpzSN}\right)$, we observe
$$
-f = \Div s = \Div(s_M+s_N) = \Div(s_M).
$$
Since $s_M\in M$, there exists a unique class $[v]\in\WpRN/\rbm_p$ such that $s_M = \Grads[v].$

\medskip

\textbf{Part III -- Uniqueness.}
If there exists another solution $s'\in M\cap\DpzSN$, $s_M-s'$ satisfies $\Div(s_M-s') = 0$, which implies $s_M-s'\in N$. Since $s_M-s'$ is also in $M$ and $M\cap N = {0}$, $s_M-s'$ vanishes, ensuring solution's uniqueness.

\medskip

\textbf{Part IV -- Minimal norm.} Let $s_f\in M$ ($\Pi s_f = 0$) and $s\in [s]$ be such that $s-s_f = \Pi s$. Due to the strict convexity of the norm, we have
$$\|s\|_{\LpSN}=\|s_f+\Pi s\|_{\LpSN}\geq\|s_f\|_{\LpSN}.$$

\end{proof}

\begin{framed}
\begin{theorem}
Let $(\Omega,\mathcal{F},\mu)$ be a measure space, $1 \leq p < \infty$ and $s\in\LpSN$. Let $\mathfrak{S}\subset\SN$ be a finite set and let $\norm[F]{\cdot}$ denote the Frobenius norm. Then, there exists at least a measurable function $\tilde{s}^*:\Omega\to\mathfrak{S}$ such that
for $\mu$-almost every $\omega\in\Omega$,
$$\tilde{s}^*(\omega)\in\arg\left\lbrace\min_{\mathfrak{s} \in \mathfrak{S}} \norm[F]{s(\omega) - \mathfrak{s}}\right\rbrace.$$
Moreover, if the tie-set, i.e., 
$$\Omega_T := \bigcup_{\substack{\mathfrak{s}_i, \mathfrak{s}_j\in\mathfrak{S} \\ \mathfrak{s}_i\neq \mathfrak{s}_j}}\lbrace \omega \in \Omega : \norm[F]{s(\omega)-\mathfrak{s}_i}=\norm[F]{s(\omega)-\mathfrak{s}_j}\rbrace,$$
satisfies $\mu(\Omega_T) = 0$, then $\tilde{s}^*$ is unique $\mu$-almost everywhere.
\end{theorem}
\end{framed}

\begin{proof}
\textbf{Part I -- Pointwise minimization.}
Since $\mathfrak{S}$ is finite, the minimization problem at each point $\omega\in\Omega$ is a search over a discrete set. For a fixed $\omega$, we define the function $g_\omega(\mathfrak{s}) = \norm[F]{s(\omega)-\mathfrak{s}}$.
On a finite set, the ``minimization'' simply reduces to
$\min\left\lbrace g_\omega(\mathfrak{s}_1),...,g_\omega(\mathfrak{s}_{|\mathfrak{S}|})\right\rbrace.$
The minimum of finitely many real numbers always exists and is attained (not necessarily unique).\\

\textbf{Part II -- Existence.}
Recall first that $\SN$ is a separable finite-dimensional normed space (hence Polish) and with its well-defined Borel $\sigma$-algebra . Since a minimizer exists for each $\omega\in\Omega$, we can define a mapping $\tilde{s}^*:\Omega\to\mathfrak{S}$. For each $\omega$, choose $\mathfrak{s}_i$ with the smallest index among all elements of $\arg\left\lbrace\min_{\mathfrak{s} \in \mathfrak{S}} \norm[F]{s(\omega) - \mathfrak{s}}\right\rbrace$. Because $s$ is a measurable function (as it belongs to $\LpSN$) and $\mathfrak{S}$ is finite, the selection of the nearest point in $\mathfrak{S}$ results in a measurable simple function. Thus, the measurable selection $\tilde{s}^*$ exists and belongs to $L^\infty(\Omega;\SN)$, and by extension to $\LpSN$.\\

\textbf{Part III -- Uniqueness.}
For distinct $\mathfrak{s}_i,\mathfrak{s}_j \in\mathfrak{S}$, define the measurable function $f_{ij}:= \norm[F]{s(\omega)-\mathfrak{s}_i}-\norm[F]{s(\omega)-\mathfrak{s}_j}$. Then the set of points where $\mathfrak{s}_i$ and $\mathfrak{s}_j$ tie is $\Omega_{T_{ij}} = f_{ij}^{-1}\lbrace 0\rbrace$. Because $\lbrace 0\rbrace$ is a Borel set on $\R$ and $f_{ij}$ is measurable, each $\Omega_{T_{ij}}\in\mathcal{F}$ is measurable. Thus, the full tie-set is measurable. The minimizer fails to be unique exactly on $\Omega_T$. If $\mu(\Omega_T) = 0$, the set of points where the $\arg\left\lbrace\min_{\mathfrak{s} \in \mathfrak{S}} \norm[F]{s(\omega) - \mathfrak{s}}\right\rbrace$ is not a singleton is negligible. Therefore, the selection $\tilde{s}^*$ is unique $\mu$-almost everywhere.\\

\textbf{Part IV -- Voronoi interpretation and measurable partition.} Because $\mathfrak{S}\subset\SN$ is finite, it induces a Voronoi tessellation of $\SN$  with respect to $\norm[F]{\cdot}$. For each $\mathfrak{s}_i\in\mathfrak{S}$, define its Frobenius-Voronoi cell
$$
V_i:=\left\lbrace\sigma\in\SN:\norm[F]{\sigma-\mathfrak{s}_i}\leq\norm[F]{\sigma-\mathfrak{s}_j}~\forall~\mathfrak{s}_j\in\mathfrak{S}\right\rbrace
=
\bigcap_{\mathfrak{s}_j\in\mathfrak{S}}
\left\lbrace\sigma\in\SN:\norm[F]{\sigma-\mathfrak{s}_i}-\norm[F]{\sigma-\mathfrak{s}_j}\leq 0\right\rbrace.
$$
Each $V_i$ is a closed Borel set in $\SN$, since it is a finite intersection of preimages of $(-\infty,0]$ under continuous maps $\sigma\mapsto\norm[F]{\sigma-\mathfrak{s}_i}-\norm[F]{\sigma-\mathfrak{s}_j}$. Consequently, the preimage $\Omega_i:=s^{-1}(V_i)=\lbrace\omega\in\Omega:s(\omega)\in V_i\rbrace$ is measurable for every $i$. On $\Omega\setminus\Omega_T$, the cells are disjoint, so $\{\Omega_i\}_{i=1}^{|\mathfrak{S}|}$ 
form a measurable partition (up to null set) of $\Omega$: for $\mu$-almost every $\omega$, the selector equals the unique label of the Frobenius-Voronoi cell containing $s(\omega)$, i.e.,
$$\tilde{s}^*(\omega) = \mathfrak{s}_i
\quad\Longleftrightarrow\quad
\omega\in \Omega_i\mathrm{~and~}\omega\notin \Omega_T.$$
\end{proof}

%\section{Conclusions}

\section{Acknowledgments}

CGG gratefully acknowledges the financial support from the European Research Council through the ERC Consolidator Grant ``DATA-DRIVEN OFFSHORE'' (Project ID 101083157). CGG also gratefully acknowledges Prof. Leandro Cagliero for insightful discussions on the measure‑theoretic aspects addressed in this work. FAR and KK would like to acknowledge the financial support from the Vista Center for Modeling of Coupled Subsurface Dynamics (VISTA CSD). KK would like to acknowledge the Center for Sustainable Subsurface Resources (CSSR). FAR wants to thank the support from the project MUPSI, CETP-2023-00298.

\bibliographystyle{plain}
\bibliography{references} 

\end{document}